\documentclass[12pt]{article}
\usepackage[cp1251]{inputenc}
\usepackage[english]{babel}
\usepackage{amssymb}
\usepackage{amsmath}
\usepackage{amsfonts}
\usepackage{amsthm}
\usepackage{graphicx}
\newtheorem{theorem}{Theorem}[section]
\newtheorem{definition}{Definition}[section]

\newtheorem{corollary}{Corollary}[section]
\newtheorem{remark}{Remark}[section]
\newtheorem{tverd}{Proposition}[section]

\begin{document}

\begin{center}
\textbf{Asymptotic  expansion for a solution of an  elliptic boundary-value problem in
a thin cascade domain}
\footnote{Please cite this article in press as:  Klevtsovskiy A.V., \&  Mel'nyk T.A.
Asymptotic expansion for a solution of an elliptic boundary-value problem in a thin cascade domain. Nonlinear Oscillations (2013), {\bf 2}}
\end{center}

\bigskip

\begin{center}
A. V. Klevtsovskiy, \ \ \ T.~A.~Mel'nyk\footnote{
Taras Shevchenko National University of Kyiv,
Faculty of Mathematics and Mechanics,
Department of Mathematical Physics,
Volodymyrska str. 64,\ 01601 Kyiv,  \ Ukraine.
\ E-mail: melnyk@imath.kiev.ua}
\end{center}

\bigskip

\begin{abstract}
Asymptotic expansion is constructed and justified for the solution to a nonuniform Neumann boundary-value problem for the Poisson equation
with the right-hand side that depends both on longitudinal and transversal variables in a thin cascade domain. Asymptotic energetic and uniform pointwise estimates for the difference between the solution of the initial problem and the solution of the corresponding limiting problem are proved.
\end{abstract}

\bigskip

{\bf Key words:} \  asymptotic expansion; thin domain; thin cascade domain

\medskip

{\bf MOS subject classification:} \  35C20, 35B40, 35J05, 74K30

\section{Introduction}

There are many articles and books (see, e.g., \cite{G62}-\cite{N-book-02}) devoted to boundary-value problems in thin domains (one of linear sizes of such a domain is substantially smaller than the others). The reason for such popularity of these problems is the wide possibilities of application of results in applied problems.
In spite of a huge progress of computational tools it is impossible to find acceptable numerical solutions of boundary-value problems in such areas because very small thickness of a thin domain naturally leads to a lengthening of computation time and significantly complicates
 the maintenance of an acceptable level of accuracy. Thus the main method of such researches is asymptotic analysis. The aim of this analysis is to develop rigorous asymptotic methods for boundary-value problems in thin domains.

In recent years, the development of modern technologies of production of porous, composite, and other microinhomogeneous
materials and biological structures has stimulated a significant interest in the investigation of
boundary-value problems  in thin domains of more complex structures:
thin perforated domains with rapidly varying thickness and different limit dimensions  \cite{Mel_Pop_MatSb,Mel_Pop_book},
thin perforated domains with rapidly varying thickness \cite{M91d,Mel_Pop_MatSb,Chech_Pichug,Kohn-V,ArrietaCarSilPer2,ArrietaCarSilPer},
junctions of thin domains \cite{NazPlam,Panasenko,Naz96,Gaudiello-Zapp,Gaudiello-Kolpakov},
thick junctions of thin domains \cite{MN94,ZAA99,BlanGau03,BlGaMe08}, periodic grids and frames \cite{CiorPaulin,ZhikovPast}.

Research of various physical and biological processes in channels are topical for many branches of science (see, e.g.,
\cite{Borysiuk2007,Borysiuk2010} and references indicated there).
 Of great interest to researchers is the appearance of different effects (such as sticking to walls, stenosis) in close vicinity to local irregularities
 (widening or narrowing) of geometry of channels. In \cite{Borysiuk2007,Borysiuk2010} the author has summarized results of recent theoretical, experimental and numerical studies of flows and wall pressure fluctuations in channels with different types of narrowing.


\begin{figure}[htbp]
\begin{center}
\includegraphics[width=15cm]{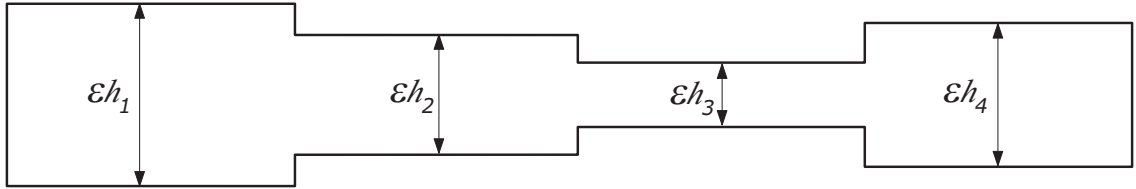}
\caption{Thin cascade domain }
\end{center}
\end{figure}

In \cite{Gaudiello-Kolpakov} by methods of formal asymptotic analysis, the limiting problem was obtained for a homogeneous Neumann problem for the Poisson equation with the right-hand side, which depends only of one longitudinal variable, in junctions of thin domains.
The authors showed that the local geometric irregularity in the joint zone does not affect the view of the limiting problem.
But, the convergence theorem and asymptotic estimates have not been proven.

It should be stressed that the error estimate and convergence rate are very important both in justifying the adequacy of one-dimensional (two-dimensional) models to real thin three-dimensional models, and in the study of boundary effects and effects of local (internal) inhomogeneities in  applied mechanics. Such estimates can be proved by developing of new asymptotic methods.

In this paper we begin to develop asymptotic methods for boundary-value problems in
{\it thin cascade domains}, which are cascade connections of thin domains with different thickness (see Fig.~1).
 To construct the formal asymptotic expansion we have generalized asymptotic method for thin domains with constant thickness proposed in monograph \cite{N-book-02}. In particular, we introduced additional inner boundary-layer asymptotic expansion in a neighborhood of the joint of thin domains and studied its properties. Thus, the asymptotics for the solution consists of three parts:  regular part, boundary part near the extreme vertical sides and inner part in in a neighborhood of the joint zone.

 It is understood that there is no principal  difference between the construction of the asymptotics for the solution to a boundary-value problem in a thin cascade domain consisting of two thin domains of varying thickness, and in a thin cascade domain consisting of $n$ thin domains of different thickness. Therefore, we consider a model two-stage thin cascade  in this paper.  Also, for simplicity, we consider two-dimensional case.

 The aim of this paper is to construct and justify the asymptotic expansion for the solution to a nonuniform
 Neumann boundary-value problem for the Poisson equation with the right-hand side
 that depends both on longitudinal and transversal variables in a thin cascade domain, which consist of two thin rectangles of different thicknesses  $\varepsilon\, h_1$ and   $\varepsilon\, h_2$.

The paper is organized as follows. In Section 2 we construct formal asymptotic expansion for the solution to  the problem  (\ref{probl}).
Section 3 is devoted  to the justification of the asymptotics  (theorem~\ref{mainTheorem}) and to the proving of the asymptotic estimates for the main terms of the asymptotics (corollary~\ref{Corollary}). In the fourth section we analyze the results obtained and show possible generalizations.

 \subsection{Statement of the problem}

A model thin cascade domain $\Omega_\varepsilon$ consists of two thin rectangles
$$
\Omega_\varepsilon^{(1)}=\left((-1,0)\times\Upsilon_\varepsilon^{(1)}\right)\quad \text{and} \quad \Omega_\varepsilon^{(2)}=\left((0,1)\times\Upsilon_\varepsilon^{(2)}\right),
$$
where $ \Upsilon_\varepsilon^{(i)}=\left(-\varepsilon\frac{h_i}{2},\varepsilon\frac{h_i}{2}\right), \ i=1,2;$ \ $\varepsilon$ is a small parameter; $h_1$ and $h_2$ are fixed positive constants, $h_2 < h_1$ (see Fig.~2).

\begin{figure}[htbp]
\begin{center}
\includegraphics[width=14cm]{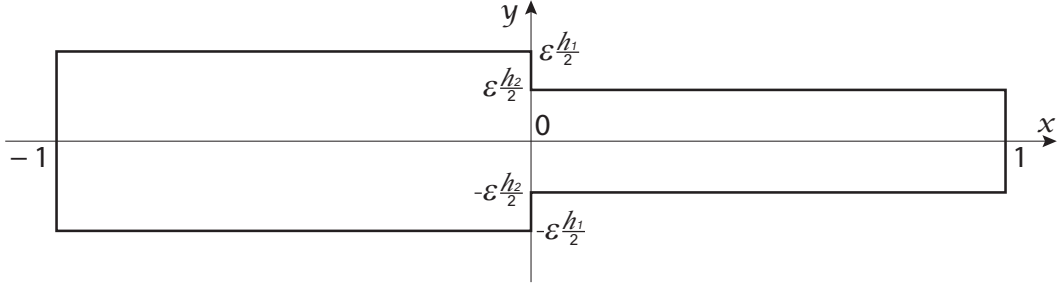}
\caption{Model thin cascade domain $\Omega_\varepsilon$}
\end{center}
\end{figure}

In $\Omega_\varepsilon=\big((-1,0)\times\Upsilon_\varepsilon^{(1)}\big)\bigcup \big([0,1)\times\Upsilon_\varepsilon^{(2)}\big)$
we consider the following mixed boundary-value problem
\begin{equation}\label{probl}
\left\{\begin{array}{rcll}
-\Delta{u_\varepsilon}(x,y)                                                & = & f(x,\frac{y}{\varepsilon}),        & (x,y)\in\Omega_\varepsilon,                  \\[2mm]
-\partial_y{u_\varepsilon}(x,y)|_{y=\pm{\varepsilon\frac{h_i}{2}}}         & = & {\varepsilon\varphi_\pm^{(i)}}(x), & x\in{I_i},\ i=1,2,                           \\[2mm]
u_\varepsilon(-1,y)                                                        & = &
0,                                 & y\in\Upsilon_\varepsilon^{(1)},              \\[2mm]
u_\varepsilon(1,y)                                                         & = &
0,                                 & y\in\Upsilon_\varepsilon^{(2)},              \\[2mm]
\partial_x{u_\varepsilon}(x,y)|_{x=0}                                      & = &
0,                                 & y\in\Upsilon_\varepsilon^{(1)}\backslash\Upsilon_\varepsilon^{(2)},               \\[2mm]
[u_\varepsilon]|_{x=0}                                                     & = &
0,                                 & y\in\Upsilon_\varepsilon^{(2)},              \\[2mm]
[\partial_x u_\varepsilon]|_{x=0}                                          & = &
0,                                 & y\in\Upsilon_\varepsilon^{(2)},
\end{array}\right.
\end{equation}
\\[3mm]
where $I_1=(-1,0),\,I_2=(0,1)$, $[u]|_{x=0}=u(x,y)|_{x=0+}- u(x,y)|_{x=0-}$ is the jump of function~$u,$
$\partial_x=\partial/\partial{x},$  $\partial^2_{xx}=\partial^2/\partial{x}^2.$
We assume that functions $f$ and $\{\varphi_\pm^{(i)}\}$ are smooth functions in their domains of definition.

From the theory of linear boundary-value problems it follows that for every fixed $\varepsilon$  there exists a unique generalized solution $u_\varepsilon$ to problem (\ref{probl}), i.e., $u_\varepsilon \in H^1(\Omega_\varepsilon),$ its traces on the vertical sides of
$\Omega_\varepsilon$ are equal to 0 $(u_\varepsilon|_{x=\pm1}=0)$ and the following integral identity
\begin{equation}\label{int-identity}
\int_{\Omega_\varepsilon} \nabla u_\varepsilon \cdot \nabla \psi \, dx\,dy =\int_{\Omega_\varepsilon} f \,\psi \, dx\,dy
 \mp \varepsilon  \sum_{i=1}^2 \int_{I_i} \varphi^{(i)}_{\pm} \, \psi\, dx
\end{equation}
holds for any function $\psi\in H^1(\Omega_\varepsilon)$ such that $\psi|_{x=\pm1}=0.$

\begin{remark}
In right-hand side of identity (\ref{int-identity}) the following  brief record
$$
 \mp \varepsilon  \sum_{i=1}^2 \int_{I_i} \varphi^{(i)}_{\pm} \, \psi\, dx :=
 - \varepsilon  \sum_{i=1}^2 \int_{I_i} \varphi^{(i)}_{+} \, \psi\, dx +
 \varepsilon  \sum_{i=1}^2 \int_{I_i} \varphi^{(i)}_{-} \, \psi\, dx,
$$
is introduced that will be used further.
\end{remark}

Our aim is to construct and justify asymptotic expansion of the solution  $u_\varepsilon$ as  $\varepsilon \to 0$.


\section{Formal construction of asymptotic series}

 \subsection{Regular part of asymptotics}

Regular part of the asymptotics is sought in the form
\begin{equation}\label{regul}
u_\infty^{(i)}:=\sum\limits_{k=2}^{+\infty}\varepsilon^{k}\left(u_k^{(i)}(x,\frac{y}{\varepsilon})+\varepsilon^{-2}\omega_k^{(i)}(x)\right), \quad (x,y)\in\Omega^{(i)}_\varepsilon, \quad i=1,2.
\end{equation}
Formally substituting the series (\ref{regul}) into the differential equation and into the first boundary condition of problem (\ref{probl}), we obtain:
$$
-\sum\limits_{k=2}^{+\infty}\varepsilon^{k}\partial^2_{xx}{u}_k^{(i)}(x,\eta) -\sum\limits_{k=2}^{+\infty}\varepsilon^{k-2}\partial^2_{\eta\eta}{u}_k^{(i)}(x,\eta) -\sum\limits_{k=2}^{+\infty}\varepsilon^{k-2}\frac{d^2\omega_k^{(i)}}{dx^2}(x) \approx f(x,\eta), \quad \eta=\frac{y}{\varepsilon},
$$
$$
-\sum\limits_{k=2}^{+\infty}\varepsilon^{k}\partial_\eta{u}_k^{(i)}(x,\pm\frac{h_i}{2}) \approx \varepsilon^2\varphi_\pm^{(i)}(x).
$$
Equating coefficients at the same degrees of $\varepsilon$, we derive recurrence relations to determine coefficients $\{ u_k^{(i)} \}.$
 Let us consider the problem for $u_2^{(i)}:$
\begin{equation}\label{regul_probl_2}
\left\{\begin{array}{rcl}
-\partial_{\eta\eta}^2{u}_2^{(i)}(x,\eta)          & = & f(x,\eta)+\dfrac{d^{\,2}\omega_2^{(i)}}{dx^2}(x), \quad \eta\in\Upsilon_i,\\[2mm]
-\partial_{\eta}u_2^{(i)}(x,\eta)|_{\eta=\pm\frac{h_i}{2}} & = & \varphi_\pm^{(i)}(x), \quad x\in I_i \\[2mm]
\langle u_2^{(i)}(x,\cdot) \rangle_{\Upsilon_i}   & = & 0, \quad x\in I_i.
\end{array}\right.
\end{equation}
Here $\Upsilon_i=\left(-\frac{h_i}{2},\frac{h_i}{2}\right), $ \
$\langle u(x,\cdot) \rangle_{\Upsilon_i} :=  \int_{\Upsilon_i}u (x,\eta)d\eta,$ $\ i=1,2.$
For each number $i$ this is a Neumann problem for the ordinary differential equation with respect to the variable $\eta\in\Upsilon_i$;  variable $x$ is considered here as a parameter. The last relation is added for the uniqueness of the solution.
The solvability condition for the problem (\ref{regul_probl_2}) gives us the differential equation
for function $\omega_2^{(i)}:$
\begin{equation}\label{omega_probl_2}
- h_i\frac{d^2\omega_2^{(i)}}{dx^2}(x) = \int_{\Upsilon_i}f(x,\eta)d\eta - \varphi_+^{(i)}(x) + \varphi_-^{(i)}(x), \quad x\in I_i \quad (i=1, 2).
\end{equation}
Let $\omega_2^{(i)}$ be a solution of the differential equation (\ref{omega_probl_2}) (boundary conditions for this differential equation will be determined later). Then the solution of problem (\ref{regul_probl_2}) is uniquely defined.

To determine the coefficients  $u_3^{(i)}, \  i=1, 2,$ we obtain the following problems
\begin{equation}\label{regul_probl_3}
\left\{\begin{array}{rcl}
-\partial_{\eta\eta}^2{u}_3^{(i)}(x,\eta)               & = & \dfrac{d^{\,2}\omega_3^{(i)}}{dx^2}(x) , \quad \eta\in\Upsilon_i, \\[2mm]
-\partial_{\eta}u_3^{(i)}(x,\eta)|_{\eta=\pm\frac{h_i}{2}} & = & 0,    \quad x\in I_i                        \\[2mm]
\langle u_3^{(i)}(x,\cdot) \rangle_{\Upsilon_i}  & = & 0, \quad x\in I_i.
\end{array}\right.
\end{equation}
Repeating previous arguments, we find that ${u}_3^{(i)}\equiv 0$  and
$\dfrac{d^{\,2}\omega_3^{(i)}}{dx^2}(x)=0$ $x\in I_i,$ $i=1, 2.$

Let us consider boundary-value problems for the functions $u_k^{(i)}, \ k\geq 4, \ i=1, 2 :$
\begin{equation}\label{regul_probl_k}
\left\{\begin{array}{rcl}
-\partial_{\eta\eta}^2{u}_k^{(i)}(x,\eta)    & = & \dfrac{d^{\,2}\omega_k^{(i)}}{dx^2}(x) + \partial_{xx}^2{u}_{k-2}^{(i)}(x,\eta), \quad \eta\in\Upsilon_i, \\[2mm]
-\partial_{\eta}u_k^{(i)}(x,\eta)|_{\eta=\pm\frac{h_i}{2}} & = & 0,     \quad x\in I_i                       \\[2mm]
\langle u_k^{(i)}(x,\cdot) \rangle_{\Upsilon_i}   & = & 0, \quad x\in I_i.
\end{array}\right.
\end{equation}
Assume that we determined all coefficients $u_2^{(i)},\dots,u_{k-1}^{(i)},\omega_2^{(i)},\dots,\omega_{k-1}^{(i)}$ of the expansion (\ref{regul}).
Then from the solvability condition for problem (\ref{regul_probl_k}) it follows that
$$
h_i\frac{d^2\omega_k^{(i)}}{dx^2}(x)=-\int_{\Upsilon_i}\partial_x^2{u}_{k-2}(x,\eta)d\eta= -\partial_x^2\left(\int_{\Upsilon_i}u_{k-2}^{(i)}(x,\eta)d\eta\right)=0,
$$
i.e., $\omega^{(i)}_k$ is a linear function and
\begin{equation}\label{omega_probl_k}
\frac{d^2\omega_k^{(i)}}{dx^2}(x)=0, \quad x\in I_i.
\end{equation}

\begin{remark}
Till now boundary conditions for differential equations (\ref{omega_probl_2}) and (\ref{omega_probl_k}) are unknown. They will be found in the following stages of construction of the asymptotics.
\end{remark}

Thus the solution to problem  (\ref{regul_probl_k}) is uniquely determined, and hence the recurrent procedure for determining
the coefficients of series (\ref{regul}) is uniquely solved.

\begin{remark}
From the recurrent procedure of problems (\ref{regul_probl_k}) it is easy to check that
$u^{(i)}_{2p+1}$ are identically equal to 0 for odd indexes  $k=2p+1, \, p\in\Bbb{N},$ $ i=1,2.$
\end{remark}


 \subsection{Boundary asymptotics near vertical sides of  domain $\Omega_\varepsilon$}

In the previous section the regular asymptotic expansion, which takes into account the right-hand side of the differential equation of problem (\ref{probl}) and the boundary conditions on the horizontal sides of the thin cascade domain $\Omega_\varepsilon,$
was constructed. In this section we will build the boundary parts of the asymptotics, which neutralize the residuals from the regular parts
 of the asymptotics both on the left side of $\Omega_\varepsilon^{(1)}$ and the right one of $\Omega_\varepsilon^{(2)}.$

In a neighborhood of the left vertical side of $\Omega_\varepsilon^{(1)}$ we seek  the asymptotic expansion for the solution in the form
\begin{equation}\label{prim+}
\Pi_\infty^{(1)}:=\sum\limits_{k=0}^{+\infty}\varepsilon^{k}\, \Pi_k^{(1)}\left(\frac{1+x}{\varepsilon},\frac{y}{\varepsilon}\right).
\end{equation}
Substituting (\ref{prim+}) into (\ref{probl}) and collecting coefficients with equal degrees of $\varepsilon$, we obtain the following mixed boundary-value problems
\begin{equation}\label{prim+probl}
 \left\{\begin{array}{rll}
  -\Delta_{\xi\eta}\Pi_k^{(1)}(\xi,\eta)=                      & 0,                  & (\xi,\eta)\in(0,+\infty)\times\Upsilon_1,                                     \\[2mm]
  -\partial_\eta\Pi_k^{(1)}(\xi,\eta)|_{\eta=\pm\frac{h_1}{2}}=& 0,                  & \xi\in(0,+\infty),                                                            \\[2mm]
  \Pi_k^{(1)}(0,\eta)=                                         & \Phi_k^{(1)}(\eta), & \eta\in\Upsilon_1,                                                            \\[2mm]
  \Pi_k^{(1)}(\xi,\eta)\to                                     & 0,                  & \xi\to+\infty,\ \ \eta\in\Upsilon_1,
 \end{array}\right.
\end{equation}\\[3mm]
where
$
\Phi_k^{(1)} =-\omega_{k+2}^{(1)}(-1), \ k=0,1,
$
\
$
\Phi_k^{(1)}(\eta)=-u_k^{(1)}(-1,\eta)-\omega_{k+2}^{(1)}(-1), \ k\geq 2.
$
Here $\xi=\frac{1+x}{\varepsilon}, \ \eta=\frac{y}{\varepsilon}.$\

 Using the method of separation of variables, we find the solution of problem (\ref{prim+probl}) at a fixed index $k:$
\begin{equation}\label{view_solution}
\Pi_k^{(1)}(\xi,\eta)=\sum\limits_{p=0}^{+\infty}\left[a_p^{(1)}e^{-\frac{2p\pi}{h_1}\xi}\cos\left(\frac{2p\pi}{h_1}\eta\right)+ b_p^{(1)}e^{-\frac{(2p+1)\pi}{h_1}\xi}\sin\left(\frac{(2p+1)\pi}{h_1}\eta\right)\right],
\end{equation}
where
$$
a_p^{(1)}=\frac{2}{h_1}\int\limits_{-\frac{h_1}{2}}^{\frac{h_1}{2}}\Phi_k^{(1)}(\eta)\cos\left(\frac{2p\pi}{h_1}\eta\right)d\eta, \quad b_p^{(1)}=\frac{2}{h_1}\int\limits_{-\frac{h_1}{2}}^{\frac{h_1}{2}}\Phi_k^{(1)}(\eta)\sin\left(\frac{(2p+1)\pi}{h_1}\eta\right)d\eta,
$$
$$
a_0^{(1)}=\frac{1}{h_1}\int\limits_{-\frac{h_1}{2}}^{\frac{h_1}{2}}\Phi_k^{(1)}(\eta)d\eta = \frac{1}{h_1}\int\limits_{-\frac{h_1}{2}}^{\frac{h_1}{2}}u_k^{(1)}(-1,\eta)d\eta-\omega_{k+2}^{(1)}(-1)=-\omega_{k+2}^{(1)}(-1).
$$

From the fourth condition in (\ref{prim+probl}) it follows that coefficient $a_0^{(1)}$ must be equal to 0. As a result,
we arrive at  boundary conditions for the functions $\{\omega_{k+2}^{(1)}\}$:
\begin{equation}\label{bv_left}
\omega_{k+2}^{(1)}(-1)=0, \quad k\in\Bbb{N}_0.
\end{equation}

In a neighborhood of the right vertical side of $\Omega_\varepsilon^{(2)}$ we seek  the asymptotic expansion for the solution in the form
\begin{equation}\label{prim-}
\Pi_\infty^{(2)}:=\sum\limits_{k=0}^{+\infty}\varepsilon^{k}\Pi_k^{(2)}\left(\frac{1-x}{\varepsilon},\frac{y}{\varepsilon}\right).
\end{equation}
To determine the coefficients $\{\Pi_k^{(2)}\}_{k\in\Bbb{N}_0}$ we get  the following boundary-value problems:
\begin{equation}\label{prim-probl}
 \left\{\begin{array}{rll}
  -\Delta_{\xi^*\eta}\Pi_k^{(2)}(\xi^*,\eta)=                    & 0,                  & (\xi^*,\eta)\in(0,+\infty)\times\Upsilon_2,                                     \\[2mm]
  -\partial_\eta\Pi_k^{(2)}(\xi^*,\eta)|_{\eta=\pm\frac{h_2}{2}}=& 0,                  & \xi^*\in(0,+\infty),                                                           \\[2mm]
  \Pi_k^{(1)}(0,\eta)=                                           & \Phi_k^{(2)}(\eta), & \eta\in\Upsilon_2,                                                              \\[2mm]
  \Pi_k^{(1)}(\xi^*,\eta)\to                                     & 0,                  & \xi^*\to+\infty, \ \ \eta\in\Upsilon_2,
 \end{array}\right.
\end{equation}\\[3mm]
where
$
\Phi_k^{(2)} =-\omega_{k+2}^{(2)}(1), \ k=0,1;
$
\
$
\Phi_k^{(2)}(\eta)=-u_k^{(2)}(1,\eta)-\omega_{k+2}^{(2)}(1), \ k\geq 2; \
$
$\xi^*=\frac{1-x}{\varepsilon},\eta=\frac{y}{\varepsilon}.$

Similarly we find the solution of the problem (\ref{prim-probl}) at a fixed index $k:$
\begin{equation}\label{view_solution2}
\Pi_k^{(2)}(\xi^*,\eta)=\sum\limits_{p=0}^{+\infty}\left[a_p^{(2)}e^{-\frac{2p\pi}{h_2}\xi^*}\cos\left(\frac{2p\pi}{h_2}\eta\right)+ b_p^{(2)}e^{-\frac{(2p+1)\pi}{h_2}\xi^*}\sin\left(\frac{(2p+1)\pi}{h_2}\eta\right)\right],
\end{equation}
where
$$
a_p^{(2)}=\frac{2}{h_2}\int\limits_{-\frac{h_2}{2}}^{\frac{h_2}{2}}\Phi_k^{(2)}(\eta)\cos\left(\frac{2p\pi}{h_2}\eta\right)d\eta, \quad b_p^{(2)}=\frac{2}{h_2}\int\limits_{-\frac{h_2}{2}}^{\frac{h_2}{2}}\Phi_k^{(2)}(\eta)\sin\left(\frac{(2p+1)\pi}{h_2}\eta\right)d\eta,
$$
$$
a_0^{(2)}=\frac{1}{h_2}\int\limits_{-\frac{h_2}{2}}^{\frac{h_2}{2}}\Phi_k^{(2)}(\eta)d\eta = \frac{1}{h_2}\int\limits_{-\frac{h_2}{2}}^{\frac{h_2}{2}}u_k^{(2)}(1,\eta)d\eta-\omega_{k+2}^{(2)}(1)=-\omega_{k+2}^{(2)}(1).
$$
From the fourth condition in (\ref{prim-probl}) it follows that coefficient $a_0^{(2)}$ must be equal to 0. It is possible when
\begin{equation}\label{bv_right}
\omega_{k+2}^{(2)}(1)=0, \quad k\in\Bbb{N}_0.
\end{equation}

\begin{remark}
Since $u_{k}^{(i)}\equiv0$ for $k=2p+1, \, p\in\Bbb{N}$, functions $\Phi_{2p+1}^{(i)}=0, \  p\in\Bbb{N}.$ As a result,
$$
\Pi_{0}^{(i)}\equiv0, \ \Pi_{2p-1}^{(i)}\equiv0, \quad p\in\Bbb{N}, \quad i=1,2.
$$
 Moreover, from representation (\ref{view_solution}) and (\ref{view_solution2}) it follow the following asymptotic relations
\begin{equation}\label{as_estimates}
 \Pi_k^{(1)}(\xi,\eta)=  {\cal O}(\exp(- \tfrac{\pi}{h_1}\xi)) \ \mbox{as} \ \xi\to+\infty, \quad
 \Pi_k^{(2)}(\xi^*,\eta)=  {\cal O}(\exp(- \tfrac{\pi}{h_2}\xi)) \ \mbox{as} \ \xi^*\to+\infty.
\end{equation}
\end{remark}

Equalities  (\ref{bv_left}) and (\ref{bv_right}) set boundary conditions at points $-1$ and $1$  for all functions $\{\omega_k^{(1)}\}$ and $\{\omega_k^{(2)}\}$, respectively. We find out conditions  for these functions at point $0$  in the next section.


 \subsection{Inner boundary part of the asymptotics}

Let us consider what happens with the regular parts of the asymptotics in the join zone of  two thin domains $\Omega_\varepsilon^{(1)}$ and $\Omega_\varepsilon^{(2)}$. Formally substituting the regular parts $u_\infty^{(1)}$ and $ u_\infty^{(2)} $ into the transmission conditions of  problem (\ref{probl}), we obtain the following relations:
\begin{equation}\label{jumpF}
\sum\limits_{k=2}^{+\infty}\varepsilon^{k}\left(u_k^{(1)}(0,\frac{y}{\varepsilon})+\varepsilon^{-2}\omega_k^{(1)}(0)\right)\approx \sum\limits_{k=2}^{+\infty}\varepsilon^{k}\left(u_k^{(2)}(0,\frac{y}{\varepsilon})+\varepsilon^{-2}\omega_k^{(2)}(0)\right),
\end{equation}
\begin{equation}\label{jumpD}
\sum\limits_{k=2}^{+\infty}\varepsilon^{k}\left(\partial_x u_k^{(1)}(0,\frac{y}{\varepsilon})+\varepsilon^{-2}\frac{d\omega_k^{(1)}}{dx}(0)\right)\approx \sum\limits_{k=2}^{+\infty}\varepsilon^{k}\left(\partial_x u_k^{(2)}(0,\frac{y}{\varepsilon})+\varepsilon^{-2}\frac{d\omega_k^{(2)}}{dx}(0)\right).
\end{equation}
Equating, for example, the corresponding coefficients with the  equal degrees of $\varepsilon$ in (\ref{jumpF}), we get
$$
\omega^{(2)}_k(0)=\omega^{(1)}_k(0), \quad \mbox{when} \ k=2 \ \mbox{ and } \  k=2p+1;
$$
\begin{equation}\label{jumpF_k}
u^{(2)}_k(0,\frac{y}{\varepsilon}) - u^{(1)}_k(0,\frac{y}{\varepsilon}) = \omega^{(2)}_{k+2}(0)-\omega^{(1)}_{k+2}(0), \quad \mbox{when} \ k=2p, \ p\in\Bbb{N}.
\end{equation}
The left-hand side of (\ref{jumpF_k}) is a known quantity that  depends on the "rapid" variable $\frac{y}{\varepsilon}$ and not necessarily equal to 0. Thus it is impossible to choose the constant  $\omega^{(2)}_{k+2}(0)-\omega^{(1)}_{k+2}(0)$ such that equality (\ref{jumpF_k}) is satisfied.

Therefore, we should  introduce an additional inner asymptotic expansion in a neighborhood of the joint zone  to remove residuals that depend on the "rapid" variable  $\frac{y}{\varepsilon}$ in the transmission conditions of problem (\ref{probl}) at the join zone of  two thin domains $\Omega_\varepsilon^{(1)}$ and $\Omega_\varepsilon^{(2)}.$

Inner expansion is  sought in the form:
\begin{equation}\label{junc}
N_\infty=\sum\limits_{k=1}^{+\infty}\varepsilon^k N_k\left(\frac{x}{\varepsilon},\frac{y}{\varepsilon}\right).
\end{equation}

Passing to the coordinates $\xi=\frac{x}{\varepsilon},\ \eta=\frac{y}{\varepsilon}$ in a neighborhood of the joint zone and then forwarding  the parameter $\varepsilon$ to 0, we obtain the following unbounded domain
$$
\Xi=\big((-\infty,0)\times\Upsilon_1\big)\cup\big([0,+\infty)\times\Upsilon_2\big),
$$
which is a union of semi strips
$\Xi^{(1)}=(-\infty,0)\times\Upsilon_1$ and $\Xi^{(2)}=(0,+\infty)\times\Upsilon_2.$

Let us introduce the following notation for parts of the boundary of the domain $\Xi$:
\begin{itemize}
  \item
 $\partial\Xi_\|=\{0\}\times \big(\Upsilon_1\setminus \Upsilon_2\big)$ is the vertical parts of the boundary $\partial\Xi$,
  \item
 $\partial\Xi^{(i)}_=$ is the horizontal parts of the boundary $\partial\Xi^{(i)},$ $i=1, 2,$
  \item
 $\partial\Xi_= = \partial\Xi^{(1)}_= \cup \partial\Xi^{(2)}_=$.
\end{itemize}

Substituting  (\ref{junc}) into (\ref{probl}), taking into account residues  that  leave the regular parts of the asymptotics on the vertical sides and on the joint zone, and equating corresponding coefficients by equal degrees of $\varepsilon$,   we get the following relations for the coefficients $\{N_k\}.$

For even numbers  $k=2p,\ p\in\Bbb{N}$:
\begin{equation}\label{junc_probl_even}
\left\{\begin{array}{rcll}
-\Delta{N_{2p}}                 & = & 0            & \mbox{in} \ \Xi,            \\[2mm]
\partial_\eta{N_{2p}}           & = & 0            & \mbox{on}\ \partial\Xi_=,  \\[2mm]
\partial_\xi{N_{2p}}            & = & \Theta_{2p}  & \mbox{on}\ \partial\Xi_\|, \\[2mm]
[N_{2p}]|_{\xi=0}               & = & \Psi_{2p}    & \mbox{on}\ \Upsilon_2,
\\[2mm]
[\partial_{\xi}N_{2p}]|_{\xi=0} & = & \Phi_{2p}    & \mbox{on}\ \Upsilon_2,
\end{array}\right.
\end{equation}
where
$$
\Theta_{2p}(\eta)=-\frac{d\omega_{2p+1}^{(1)}}{dx}(0), \quad \eta\in\partial\Xi_\|,
$$
$$
\Psi_{2p}(\eta)=u_{2p}^{(1)}(0,\eta)-u_{2p}^{(2)}(0,\eta), \quad \eta\in\Upsilon_2,
$$
$$
\Phi_{2p}(\eta)=\frac{d\omega_{2p+1}^{(1)}}{dx}(0)-\frac{d\omega_{2p+1}^{(2)}}{dx}(0), \quad \eta\in\Upsilon_2.
$$

For odd numbers  $k=2p+1, \, p\in\Bbb{N}_0$:
\begin{equation}\label{junc_probl_odd}
\left\{\begin{array}{rcll}
-\Delta{N_{2p+1}}                 & = & 0             & \mbox{in} \ \Xi,            \\[2mm]
\partial_\eta{N_{2p+1}}           & = & 0             & \mbox{on}\ \partial\Xi_=,  \\[2mm]
\partial_\xi{N_{2p+1}}            & = & \Theta_{2p+1} & \mbox{on}\ \partial\Xi_\|, \\[2mm]
[N_{2p+1}]|_{\xi=0}               & = & 0             & \mbox{on}\       \Upsilon_2,
 \\[2mm]
[\partial_{\xi}N_{2p+1}]|_{\xi=0} & = & \Phi_{2p+1}   & \mbox{on}\ \Upsilon_2,
\end{array}\right.
\end{equation}
where
$$
\Theta_{2p+1}(\eta)=-\partial_x{u}_{2p}^{(1)}(0,\eta)-\frac{d\omega_{2p+2}^{(1)}}{dx}(0), \quad \eta\in\partial\Xi_\|,
$$
$$
\Phi_{2p+1}(\eta)=\partial_x{u}_{2p}^{(1)}(0,\eta)-\partial_x{u}_{2p}^{(2)}(0,\eta)+\frac{d\omega_{2p+2}^{(1)}}{dx}(0)-\frac{d\omega_{2p+2}^{(2)}}{dx}(0), \quad \eta\in\Upsilon_2,
$$
It should be noted here  that $u_0\equiv 0, \ u_1\equiv 0.$

In order to find out whether exist functions that satisfy the relations of problems (\ref{junc_probl_even}) and (\ref{junc_probl_odd}),
at first we study  the solvability the following boundary-value problem:
\begin{equation}\label{junc_probl_general}
\left\{\begin{array}{rcll}
-\Delta{N}(\xi, \eta)  & = & F(\xi, \eta),             &  \quad (\xi, \eta) \in \Xi;
\\[2mm]
\partial_\eta{N}(\xi, \eta)|_{\eta=\pm \frac{h_i}{2}}& = &\pm B^{(i)}_{\pm}(\xi),    & \quad (-1)^i \xi \in (0, +\infty), \ \ i=1, 2;
\\[2mm]
\partial_\xi{N}(\xi, \eta)|_{\xi=0} & = & G(\eta), & \quad \eta \in \Upsilon_1\setminus \Upsilon_2;
\\[2mm]
[N]|_{\xi=0}      & = & \Psi(\eta),             & \quad \eta\in  \Upsilon_2;
 \\[2mm]
[\partial_{\xi}N]|_{\xi=0} & = & \Phi(\eta),   & \quad \eta \in \Upsilon_2.
\end{array}\right.
\end{equation}

 Let $C^{\infty}_{0,\xi}(\overline{\Xi})$ be a space of infinitely differentiable functions in $\overline{\Xi}$ that are finite with respect to $\xi$, i.e.,
$$
\forall \,v\in C^{\infty}_{0,\xi}(\overline{\Xi}) \quad \exists \,R>0 \quad \forall \, (\xi,\eta)\in\overline{\Xi} \quad |\xi|\geq R: \quad v(\xi,\eta)=0.
$$
Define the following space $\mathcal{H}:=\overline{\left( C^{\infty}_{0,\xi}(\overline{\Xi}), \ \| \cdot \|_\mathcal{H} \right)}$, where
$$
\|v\|_\mathcal{H}=\sqrt{\int_\Xi|\nabla v(\xi, \eta)|^2 \, d\xi d\eta+ \int_\Xi |v(\xi, \eta)|^2 |\rho(\xi)|^2 \, d\xi d\eta} ,
$$
and the weight function $\rho(\xi)= (1+|\xi|)^{-1}, \ \xi \in \Bbb R.$

{\bf The case  $\Psi\equiv 0$.}
Let us take an arbitrary function $v\in C^{\infty}_{0,\xi}(\overline{\Xi})$, multiply the  differential equation of problem (\ref{junc_probl_general})
 by $v,$  and integrate this equality over domain $\Xi$. Using the Green-Ostrogradsky formula, we deduce the following integral identity:
\begin{multline}\label{integ_odd}
\int\limits_{\Xi}\nabla N \cdot \nabla v \, d\xi d\eta =
\int\limits_{\Xi} F \, v \, d\xi d\eta + \int^0_{-\infty} B^{(1)}_{\pm}(\xi) \, v(\xi, \pm \tfrac{h_1}{2})\, d\xi
+ \int_0^{+\infty} B^{(2)}_{\pm}(\xi) \, v(\xi, \pm \tfrac{h_2}{2})\, d\xi +
\\
+
\int\limits_{\Upsilon_1\setminus \Upsilon_2} G(\eta) \, v(0,\eta) \, d\eta
 - \int\limits_{\Upsilon_2}   \Phi(\eta) v(0,\eta)\, d\eta.
\end{multline}

\begin{definition}
A function $N$ from the space $\mathcal{H}$ is called a weak solution of problem (\ref{junc_probl_general}) if  the identity  (\ref{integ_odd}) holds for all $v\in\mathcal{H}$.
\end{definition}

From lemma 4.1, remark 4.1 and 4.2 (see \cite{ZAA99}) it follows the following proposition.

\begin{tverd}\label{tverd1}
Let $\rho^{-1} F\in L^2(\Xi),$ $\rho^{-1} B^{(2)}_{\pm} \in L^2(0,+\infty),$ $\rho^{-1} B^{(1)}_{\pm} \in L^2(-\infty,0),$
$G\in L^2(\Upsilon_1\setminus \Upsilon_2)$ and $\Phi\in L^2(\Upsilon_2).$

Then there exist a weak solution of problem (\ref{junc_probl_general}) if and only if
\begin{equation}\label{solvability}
\int\limits_{\Upsilon_2}   \Phi(\eta) \, d\eta =
\int\limits_{\Xi} F \, d\xi d\eta + \int^0_{-\infty} B^{(1)}_{\pm}(\xi) \,  d\xi
+ \int_0^{+\infty} B^{(2)}_{\pm}(\xi) \,  d\xi +
\int\limits_{\Upsilon_1\setminus \Upsilon_2} G(\eta) \,  d\eta.
\end{equation}

In addition, this solution is defined up to an additive constant and we can choose this constant
 such that there will exist a unique solution of the problem (\ref{junc_probl_general}) with the
 following differentiable asymptotics:
\begin{equation}\label{inner_asympt_general}
{\cal N}_0(\xi,\eta)=\left\{
\begin{array}{rl}
{\cal O}(\exp( \frac{\pi}{h_1}\xi)) & \mbox{when} \ \xi\to-\infty,
\\[2mm]
 d_0 + {\cal O}(\exp(-\frac{\pi}{h_2}\xi)) & \mbox{when} \ \xi\to+\infty.
\end{array}
\right.
\end{equation}

Moreover, if the functions $F, \, G, \, \Phi$ are even with respect of $\eta$ $(F, \, G, \, \Phi$ are odd with respect of $\eta)$ and
$B^{(i)}_{-}\equiv B^{(i)}_{+}, \ i=1, 2$ $(B^{(i)}_{-}\equiv - B^{(i)}_{+}, \ i=1, 2),$ then solution ${\cal N}_0$ is even (odd)
function with respect of $\eta.$ If ${\cal N}_0$ is odd function, then in (\ref{inner_asympt_general}) the constant  $d_0$ is equal to zero.
\end{tverd}

From corollary 4.1 (\cite{ZAA99}) it follows the second proposition.

\begin{tverd}\label{tverd2}
There exists a nontrivial solution ${\cal Z}_0$ of the corresponding homogeneous  problem (\ref{junc_probl_general}),
which does not belong to the space $\mathcal{H},$  and this solution has the following differentiable  asymptotics:
\begin{equation}\label{nonenergetic}
{\cal Z}_0(\xi,\eta)=\left\{
\begin{array}{rl}
\dfrac{\xi}{h_1} + C_{h_1} + {\cal O}(\exp( \frac{\pi}{h_1}\xi)) & \mbox{when} \ \xi\to-\infty,
\\[3mm]
\dfrac{\xi}{h_2} + C_{h_2} + {\cal O}(\exp(-\frac{\pi}{h_2}\xi)) & \mbox{when} \ \xi\to+\infty,
\end{array}
\right.
\end{equation}
where  $ C_{h_i}= h^{-1}_i \langle{\cal Z}_0(0,\eta)\rangle_{\Upsilon_i}, \ i=1, 2.$

Moreover, the function ${\cal Z}_0$ is even with respect of variable $\eta$ and any other solution of the homogeneous problem (\ref{junc_probl_general}), which has polynomial growth when $\xi \to \pm\infty,$  is a linear combination $\alpha_1 + \alpha_2  {\cal Z}_0,$ where $\alpha_1$ and $\alpha_2$ are some constants.
\end{tverd}

\begin{remark}\label{remark_constant}
Using the second Green-Ostrogradsky formula,  similarly as was done in remark~4.3 (\cite{ZAA99}),
constant $d_0$ from (\ref{inner_asympt_general}) can be found as follows
\begin{multline}\label{const_d_0}
d_0 =
\int\limits_{\Xi} F(\xi, \eta) \, {\cal Z}_0(\xi, \eta) \, d\xi d\eta + \int^0_{-\infty} B^{(1)}_{\pm}(\xi) \, {\cal Z}_0(\xi, \pm \tfrac{h_1}{2})\, d\xi
+ \int_0^{+\infty} B^{(2)}_{\pm}(\xi) \, {\cal Z}_0(\xi, \pm \tfrac{h_2}{2})\, d\xi +
\\
+
\int\limits_{\Upsilon_1\setminus \Upsilon_2} G(\eta) \, {\cal Z}_0(0,\eta) \, d\eta
 - \int\limits_{\Upsilon_2}   \Phi(\eta) \, {\cal Z}_0(0,\eta)\, d\eta .
\end{multline}
\end{remark}

In {\bf the general case} when $\Psi \neq 0,$ the following  substitution
$$
W=N - \chi_\delta \Psi
$$
must be done in the problem (\ref{junc_probl_general}),
where $\chi_\delta \in C^{\infty}(\Bbb{R}_+),\ 0\leq \chi_\delta \leq1$ and
$$
\chi_\delta(\xi)=
\left\{\begin{array}{ll}
0, & \text{when} \ \ \xi <0,
\\
1, & \text{when} \ \ 0\leq\xi\leq \delta,  \quad (\delta>0)
\\
0, & \text{when} \ \ \xi\geq 2\delta.
\end{array}\right.
$$
Then $[W]|_{\xi=0}=0$ and we arrive to the previous case.

\begin{definition} Let $\Psi \in H^1(\Upsilon_2).$
A function $N$ is called a weak solution of problem (\ref{junc_probl_general}), if there exists a function $W$ from $\mathcal{H}$ such that
the following integral identity
\begin{multline}\label{integ_even}
\int\limits_{\Xi}\nabla W \cdot \nabla v \, d\xi d\eta =
\int\limits_{\Xi} F \, v \, d\xi d\eta
- \int\limits_{\Xi^{(2)}}  \nabla\big(\chi_\delta(\xi) \Psi(\eta)\big)  \cdot \nabla v \, d\xi d\eta
+ \int^0_{-\infty} B^{(1)}_{\pm}(\xi) \, v(\xi, \pm \tfrac{h_1}{2})\, d\xi
+
\\
+ \int_0^{+\infty} B^{(2)}_{\pm}(\xi) \, v(\xi, \pm \tfrac{h_2}{2})\, d\xi +
\int\limits_{\Upsilon_1\setminus \Upsilon_2} G(\eta) \, v(0,\eta) \, d\eta
 - \int\limits_{\Upsilon_2}   \Phi(\eta) v(0,\eta)\, d\eta
\end{multline}
holds for all  $v\in\mathcal{H}$.
\end{definition}

\begin{remark}
It is easy to verify that in the general case, the proposition \ref{tverd1} holds with the same solvability condition
(\ref{solvability}). Moreover, the function $\Psi$ must be even (odd) in the last item of Proposition \ref{tverd1}.
\end{remark}

Now we back to problems (\ref{junc_probl_even}) and (\ref{junc_probl_odd}). From (\ref{solvability}) it follows that the following equalities
\begin{equation}\label{cond_t1}
\int\limits_{\partial\Xi_\|}\Theta_k(\eta) d\eta = \int\limits_{\Upsilon_2} \Phi_k(\eta) d\eta, \quad k\in \Bbb N,
\end{equation}
are the corresponding solvability conditions for those problems.

Taking into account that $\langle u_k^{(i)}(x,\cdot) \rangle_{\Upsilon_i}  =  0, \ \  x\in I_i, \ \ i=1, 2,$
from (\ref{cond_t1}) we  derive  the following relations for functions $\{\omega_k^{(i)}\}:$
\begin{equation}\label{transmisiont1}
h_1 \frac{d\omega_{k}^{(1)}}{dx}(0) = h_2 \frac{d\omega_{k}^{(2)}}{dx}(0), \quad k \in \Bbb N, \ \ k\geq 2.
\end{equation}
Hence, if $\{\omega_k^{(i)}\}$ satisfy (\ref{transmisiont1}), then there exist solutions of problems (\ref{junc_probl_even}) and (\ref{junc_probl_odd}). According to Proposition  \ref{tverd1}, we can uniquely choose those solutions such that  they have  the following asymptotics
\begin{equation}\label{inner_asympt}
{\cal N}_{k}(\xi,\eta)=\left\{
\begin{array}{rl}
{\cal O}(\exp( \frac{\pi}{h_1}\xi)) & \mbox{when} \ \xi\to-\infty, \\[2mm]
 d^{+}_{k} + {\cal O}(\exp(-\frac{\pi}{h_2}\xi)) & \mbox{when} \ \xi\to+\infty.
\end{array}
\right.
\end{equation}

In what follows, in asymptotic expansion (\ref{junc}) we will use the functions
$$
N_{k}(\xi,\eta)=
\left\{\begin{array}{ll}
{\cal N}_k(\xi,\eta), & \xi<0, \\[2mm]
{\cal N}_k(\xi,\eta) - d_k^+, & \xi>0,
\end{array}\right.  \quad
k\in \Bbb N.
$$
Then taking into account (\ref{inner_asympt}),  functions $\{N_{k}\}$ are  exponentially decrease  as  $\xi \to \pm \infty.$

Formally substituting the series (\ref{regul}) and (\ref{junc})
in the first transmission condition of problem (\ref{probl}), we get
$$
\left( u^{(1)}_\infty + N_\infty  \right)|_{x=0-}=\left( u^{(2)}_\infty + N_\infty  \right)|_{x=0+},
$$
from where, by virtue of the corresponding equalities in problems (\ref{junc_probl_even}) and (\ref{junc_probl_odd}),
we deduce  the following relations for the functions $\{\omega_k^{(i)}\}:$
$$
\omega_2^{(2)}(0) = \omega_2^{(1)}(0), \qquad
\omega_k^{(2)}(0) - \omega_k^{(1)}(0) = d_{k-2}^+ \ , \quad k\in\Bbb N, \ \ k \geq3.
$$

Thus, we have obtained  a sequence of boundary-value problems to determine functions $\{\omega_k^{(i)}\}.$
For functions $\omega_2^{(1)}$ and $\omega_2^{(2)}$ that form the main term of the regular asymptotic expansion (\ref{regul}),
the problem looks as follows
\begin{equation}\label{main}
\left\{\begin{array}{rcl}
- h_i\dfrac{d^2\omega_2^{(i)}}{dx^2}(x) & = &  \widehat{F}^{(i)}(x),  \qquad x\in I_i,   \ \ i=1, 2,
\\[3mm]
\omega_2^{(1)}(0) & = & \omega_2^{(2)}(0),
\\[2mm]
h_1\dfrac{d\omega_2^{(1)}}{dx}(0) & = & h_2\dfrac{d\omega_2^{(2)}}{dx}(0),
\\[4mm]
\omega_2^{(1)}(-1) & = & 0, \quad \omega_2^{(2)}(1)  \ = \ 0 ,
\end{array}\right.
\end{equation}
where
\begin{equation}\label{right-hand-side}
\widehat{F}^{(i)}(x):= \int_{\Upsilon_i}f(x,\eta)\, d\eta - \varphi_+^{(i)}(x) + \varphi_-^{(i)}(x), \quad \ x\in I_i, \ \ i=1, 2.
\end{equation}
The problem (\ref{main})  will be called  {\it  homogenized problem} for  problem (\ref{probl}).

For next functions $\{\omega_k^{(1)}, \  \omega_k^{(2)}: \  k \geq 3\}$ we get the following  problems:
\begin{equation}\label{omega_probl*}
\left\{\begin{array}{rcll}
- h_i\dfrac{d^2\omega_k^{(i)}}{dx^2}(x) & = & 0,  & x\in I_i,   \ \ i=1, 2,
\\[3mm]
\omega_k^{(1)}(0) & = & \omega_k^{(2)}(0) - d_{k-2}^+,
\\[2mm]
h_1\dfrac{d\omega_k^{(1)}}{dx}(0) & = & h_2\dfrac{d\omega_k^{(2)}}{dx}(0),
\\[4mm]
\omega_k^{(1)}(-1) & = & 0, \quad \omega_k^{(2)}(1)  \ = \ 0 .
\end{array}\right.
\end{equation}
It is easy to find the unique solution to problem (\ref{omega_probl*}) at a fixed index $k:$
\begin{equation}\label{solutions}
\omega_k^{(1)}(x) = - \frac{h_2 \, d_{k-2}^{+}}{h_1+h_2} \, (x + 1), \ \ x\in I_1; \quad
\omega_k^{(2)}(x) =  \frac{h_1\, d_{k-2}^{+}}{h_1+h_2} \, (1 - x), \ \ x\in I_2.
\end{equation}

\section{Scheme of construction of the complete asymptotics and its justification}

Let us introduce the following notation
$$
u_k\big(x,\frac{y}{\varepsilon}\big)=
\left\{\begin{array}{rl}
u_k^{(1)}\left(x,\frac{y}{\varepsilon}\right), & (x,y)\in\Omega_\varepsilon^{(1)} \\[2mm]
u_k^{(2)}\left(x,\frac{y}{\varepsilon}\right), & (x,y)\in\Omega_\varepsilon^{(2)}
\end{array}\right. ,
\quad
\omega_k(x)=
\left\{\begin{array}{rl}
\omega_k^{(1)}(x), & x\in I_1 \\[2mm]
\omega_k^{(2)}(x), & x\in I_2
\end{array}\right. , \quad k \in \Bbb N, \ k \ge 2.
$$

From homogenized problem (\ref{main}) we uniquely determine the main term of the asymptotics $\omega_2$ of the series (\ref{regul}).
Then from problems (\ref{regul_probl_2}) that can now be rewritten as
\begin{equation}\label{new_regul_probl_2}
\left\{\begin{array}{rcl}
-\partial_{\eta\eta}^2{u}_2^{(i)}(x,\eta)          & = & f(x,\eta) - h_i^{-1} \widehat{F}^{(i)}(x), \quad \eta\in\Upsilon_i,\\[2mm]
-\partial_{\eta}u_2^{(i)}(x,\eta)|_{\eta=\pm\frac{h_i}{2}} & = & \varphi_\pm^{(i)}(x), \quad x\in I_i \\[2mm]
\langle u_2^{(i)}(x,\cdot) \rangle_{\Upsilon_i}   & = & 0, \quad x\in I_i,
\end{array}\right. \qquad i=1, 2,
\end{equation}
we uniquely determine
\begin{equation}\label{solution_t}
u_2^{(i)}(x,\eta)=-\int_{-\frac{h_i}{2}}^\eta(\eta-t)\Big(f(x,t)
- h_i^{-1} \widehat{F}^{(i)}(x) \Big)dt - \eta\varphi_-^{(i)}(x) + \alpha_2^{(i)}(x),
\end{equation}
where function $\alpha_2^{(i)}$ are uniquely determined from third condition in (\ref{new_regul_probl_2}), i.e.
$$
\alpha_2^{(i)}(x) = \int_{\Upsilon_i} \int_{-\frac{h_i}{2}}^\eta(\eta-t) \, f(x,t)\, dt\, d\eta
-  6^{-1} h_i^2 \widehat{F}^{(i)}(x), \quad  i= 1, 2;
$$
functions $\widehat{F}^{(1)}$ and $\widehat{F}^{(2)}$ are given by formulas (\ref{right-hand-side}).

Now with the help of formulas (\ref{view_solution}) and (\ref{view_solution2}), we find the first terms $\Pi_2^{(1)}$ and $\Pi_2^{(2)}$ of the
boundary-asymptotic expansions (\ref{prim+})
and  (\ref{prim-}) respectively; they are solutions of problems  (\ref{prim+probl}) and (\ref{prim-probl}) that can be rewritten as follows
\begin{equation}\label{new_prim+probl}
 \left\{\begin{array}{rll}
  -\Delta_{\xi\eta}\Pi_2^{(1)}(\xi,\eta)=                      & 0,                  & (\xi,\eta)\in(0,+\infty)\times\Upsilon_1,                                     \\[2mm]
  -\partial_\eta\Pi_2^{(1)}(\xi,\eta)|_{\eta=\pm\frac{h_1}{2}}=& 0,                  & \xi\in(0,+\infty),                                                            \\[2mm]
  \Pi_2^{(1)}(0,\eta)=                                         & -u_2^{(1)}(-1,\eta), & \eta\in\Upsilon_1,                                                            \\[2mm]
  \Pi_2^{(1)}(\xi,\eta)\to                                     & 0,                  & \xi\to+\infty,\ \ \eta\in\Upsilon_1,
 \end{array}\right.
\end{equation}
\begin{equation}\label{new_prim-probl}
 \left\{\begin{array}{rll}
  -\Delta_{\xi^*\eta}\Pi_2^{(2)}(\xi^*,\eta)=                    & 0,                  & (\xi^*,\eta)\in(0,+\infty)\times\Upsilon_2,                                     \\[2mm]
  -\partial_\eta\Pi_2^{(2)}(\xi^*,\eta)|_{\eta=\pm\frac{h_2}{2}}=& 0,                  & \xi^*\in(0,+\infty),                                                           \\[2mm]
  \Pi_2^{(1)}(0,\eta)=                                           & -u_2^{(2)}(1,\eta), & \eta\in\Upsilon_2,                                                              \\[2mm]
  \Pi_2^{(1)}(\xi^*,\eta)\to                                     & 0,                  & \xi^*\to+\infty, \ \ \eta\in\Upsilon_2.
 \end{array}\right.
\end{equation}

Then we find the first term of the inner asymptotic expansion (\ref{junc})
$$
N_{1}(\xi,\eta)=
\left\{\begin{array}{ll}
{\cal N}_1(\xi,\eta), & \xi<0, \\[2mm]
{\cal N}_1(\xi,\eta) - d_1^+, & \xi>0,
\end{array}\right. ,
$$
where ${\cal N}_1$ is the unique solution of the problem (\ref{junc_probl_odd}) that can now be rewritten as
\begin{equation}\label{new_junc_probl_odd}
\left\{\begin{array}{rcll}
-\Delta{{\cal N}_{1}}                 & = & 0             & \mbox{in} \ \Xi,
\\[2mm]
\partial_\eta{{\cal N}_{1}}           & = & 0             & \mbox{on}\ \partial\Xi_=,
\\[1mm]
\partial_\xi{{\cal N}_{1}}            & = & -\dfrac{d\omega_{2}^{(1)}}{dx}(0) & \mbox{on}\ \partial\Xi_\|,
\\[3mm]
[{\cal N}_{1}]|_{\xi=0}               & = & 0             & \mbox{on}\       \Upsilon_2,
 \\[2mm]
[\partial_{\xi}{\cal N}_{1}]|_{\xi=0} & = & \dfrac{d\omega_{2}^{(1)}}{dx}(0)-\dfrac{d\omega_{2}^{(2)}}{dx}(0)   & \mbox{on}\ \Upsilon_2,
\end{array}\right.
\end{equation}
with asymptotics (\ref{inner_asympt}). Recall that constant  $d_1^+$ is also uniquely determined (see remark~\ref{remark_constant}).

Thus we have uniquely determined the first terms of the asymptotic expansions (\ref{regul}), (\ref{prim+}),  (\ref{prim-}) and (\ref{junc}).

Assume that we have determined coefficients  $\omega_2,\ldots,\omega_{2n-2},$
$u_2, \, u_4,\ldots,u_{2n-2}$ of the series (\ref{regul}),
coefficients
$\Pi^{(i)}_2,  \Pi^{(i)}_4,\ldots,\Pi^{(i)}_{2n-2}$ of the series (\ref{prim+}) and (\ref{prim-}) respectively and
coefficients
$N_1,\ldots,N_{2n-3}$ of the series (\ref{junc}).

Then, using formulas (\ref{solutions}),  we write the solution $\omega_{2n-1}$ of problem (\ref{omega_probl*})  with the constant  $d^+_{2n-3}$
 in the first transmission condition.
Further we find the coefficient
$$
N_{2n-2}(\xi,\eta)=
\left\{\begin{array}{ll}
{\cal N}_{2n-2}(\xi,\eta), & \xi<0, \\[2mm]
{\cal N}_{2n-2}(\xi,\eta) - d_{2n-2}^+, & \xi>0,
\end{array}\right. ,
$$
of the inner asymptotic expansion (\ref{junc}),
where ${\cal N}_{2n-2}$ is the unique solution of the problem  (\ref{junc_probl_odd}) that can now be rewritten as
\begin{equation}\label{junc_probl_odd_2n-2}
\left\{\begin{array}{rcll}
-\Delta{{\cal N}_{2n-2}}                 & = & 0             & \mbox{in} \ \Xi,
\\[2mm]
\partial_\eta{{\cal N}_{2n-2}}           & = & 0             & \mbox{on}\ \partial\Xi_=,
\\[1mm]
\partial_\xi{{\cal N}_{2n-2}}            & = &  \dfrac{h_2 \, d^+_{2n-3}}{h_1 + h_2} & \mbox{on}\ \partial\Xi_\|,
\\[3mm]
[{\cal N}_{2n-2}]|_{\xi=0}               & = & u_{2n-2}^{(1)}(0,\eta)-u_{2n-2}^{(2)}(0,\eta)             & \mbox{on}\       \Upsilon_2,
 \\[2mm]
[\partial_{\xi}{\cal N}_{2n-2}]|_{\xi=0} & = & \dfrac{d^+_{2n-3} (h_1 - h_2) }{h_1+h_2}   & \mbox{on}\ \Upsilon_2,
\end{array}\right.
\end{equation}
and ${\cal N}_{2n-2}$ has the asymptotics (\ref{inner_asympt}).

Knowing $d_{2n-2}^+$ and using (\ref{solutions}), we get  the solution $\omega_{2n}$ of problem (\ref{omega_probl*}).
Next coefficient
$$
N_{2n-1}(\xi,\eta)=
\left\{\begin{array}{ll}
{\cal N}_{2n-1}(\xi,\eta), & \xi<0, \\[2mm]
{\cal N}_{2n-1}(\xi,\eta) - d_{2n-1}^+, & \xi>0,
\end{array}\right. ,
$$
of the inner asymptotic expansion (\ref{junc}) is defined with the help of  solution ${\cal N}_{2n-1}$ to  problem (\ref{junc_probl_odd})
that can now  be rewritten as follows
\begin{equation}\label{new_junc_probl_odd}
\left\{\begin{array}{rcll}
-\Delta{{\cal N}_{2n-1}}                 & = & 0             & \mbox{in} \ \Xi,            \\[2mm]
\partial_\eta{{\cal N}_{2n-1}}           & = & 0             & \mbox{on}\ \partial\Xi_=,  \\[2mm]
\partial_\xi{{\cal N}_{2n-1}}            & = & -\partial_x{u}_{2n-2}^{(1)}(0,\eta) + \dfrac{h_2 \, d^+_{2n-2}}{h_1 + h_2} & \mbox{on}\ \partial\Xi_\|, \\[2mm]
[{\cal N}_{2n-1}]|_{\xi=0}               & = & 0             & \mbox{on}\       \Upsilon_2,
 \\[2mm]
[\partial_{\xi}{\cal N}_{2n-1}]|_{\xi=0} & = &  \partial_x{u}_{2n-2}^{(1)}(0,\eta)-\partial_x{u}_{2n-2}^{(2)}(0,\eta)+
\dfrac{d^+_{2n-2} (h_1 - h_2) }{h_1+h_2}
 & \mbox{on}\ \Upsilon_2.
\end{array}\right.
\end{equation}

Coefficients $u_{2n}^{(i)}, \ i=1, 2,$ are determined as  solutions of following  problems
\begin{equation}\label{new_regul_probl_k}
\left\{\begin{array}{rcl}
-\partial_{\eta\eta}^2{u}_{2n}^{(i)}(x,\eta)    & = &  \partial_{xx}^2{u}_{2n-2}^{(i)}(x,\eta), \quad \eta\in\Upsilon_i, \\[2mm]
-\partial_{\eta}u_{2n}^{(i)}(x,\eta)|_{\eta=\pm\frac{h_i}{2}} & = & 0,     \quad x\in I_i                       \\[2mm]
\langle u_{2n}^{(i)}(x,\cdot) \rangle_{\Upsilon_i}   & = & 0, \quad x\in I_i.
\end{array}\right.
\end{equation}

And finally, we find the coefficients  $\Pi_{2n}^{(1)}$ and $\Pi_{2n}^{(2)}$ of the boundary asymptotic expansions (\ref{prim+})
and (\ref{prim-}) respectively as the solutions of problems (\ref{prim+probl}) and (\ref{prim-probl}) that can be rewritten now as follows
\begin{equation}\label{new_prim+probl}
 \left\{\begin{array}{rll}
  -\Delta_{\xi\eta}\Pi_{2n}^{(1)}(\xi,\eta)=                      & 0,                  & (\xi,\eta)\in(0,+\infty)\times\Upsilon_1,                                     \\[2mm]
  -\partial_\eta\Pi_{2n}^{(1)}(\xi,\eta)|_{\eta=\pm\frac{h_1}{2}}=& 0,                  & \xi\in(0,+\infty),                                                            \\[2mm]
  \Pi_{2n}^{(1)}(0,\eta)=                                         & -u_{2n}^{(1)}(-1,\eta), & \eta\in\Upsilon_1,                                                            \\[2mm]
  \Pi_{2n}^{(1)}(\xi,\eta)\to                                     & 0,                  & \xi\to+\infty,\ \ \eta\in\Upsilon_1,
 \end{array}\right.
\end{equation}
\begin{equation}\label{new_prim-probl}
 \left\{\begin{array}{rll}
  -\Delta_{\xi^*\eta}\Pi_{2n}^{(2)}(\xi^*,\eta)=                    & 0,                  & (\xi^*,\eta)\in(0,+\infty)\times\Upsilon_2,                                     \\[2mm]
  -\partial_\eta\Pi_{2n}^{(2)}(\xi^*,\eta)|_{\eta=\pm\frac{h_2}{2}}=& 0,                  & \xi^*\in(0,+\infty),                                                           \\[2mm]
  \Pi_{2n}^{(1)}(0,\eta)=  & -u_{2n}^{(2)}(1,\eta), & \eta\in\Upsilon_2,                                                              \\[2mm]
  \Pi_{2n}^{(1)}(\xi^*,\eta)\to                                     & 0,                  & \xi^*\to+\infty, \ \ \eta\in\Upsilon_2.
 \end{array}\right.
\end{equation}

Thus, we can successively determine all coefficients of series (\ref{regul}), (\ref{prim+}),  (\ref{prim-}) and (\ref{junc}).

With the help of the series (\ref{regul}), (\ref{prim+}),  (\ref{prim-}) and (\ref{junc}) we construct the
following series
$$
\omega_2(x) + \sum\limits_{k=1}^{+\infty} \varepsilon^{2k-1}\Bigg( \omega_{2k+1}(x) + \chi^0(x) {N}_{2k-1}\left(\frac{x}{\varepsilon},\frac{y}{\varepsilon}\right) \Bigg) +
\sum\limits_{k=1}^{+\infty} \varepsilon^{2k}\Bigg( u_{2k}\left(x,\frac{y}{\varepsilon}\right) + \omega_{2k+2}(x) +
$$
\begin{equation}\label{asymp_expansion}
 + \chi^0(x){N}_{2k}\left(\frac{x}{\varepsilon},\frac{y}{\varepsilon}\right) +  \chi^-(x)\Pi^{(1)}_{2k}\Big(\frac{1+x}{\varepsilon},\frac{y}{\varepsilon}\Big) + \chi^+(x)\Pi^{(2)}_{2k}\Big(\frac{1-x}{\varepsilon},\frac{y}{\varepsilon}\Big) \Bigg), \quad (x,y)\in\Omega_\varepsilon,
\end{equation}
where $\chi^\pm, \ \chi^0$ are smooth cut-off functions defined by formulas
$$
\chi^\pm(x)=
\left\{\begin{array}{ll}
1, & \text{if} \ \ |1\mp x| \le \delta,
\\
0, & \text{if} \ \ |1\mp x| \ge 2 \delta,
\end{array}\right.
\quad
\chi^0(x)=
\left\{\begin{array}{ll}
1, & \text{if} \ \ |x| < \delta,
\\
0, & \text{if} \ \ |x|  > 2\delta.
\end{array}\right.
$$
Here $\delta$ is an arbitrary sufficiently small fixed positive number.

\bigskip

\begin{theorem}\label{mainTheorem}
 Series $(\ref{asymp_expansion})$ is the asymptotic expansion for the solution of the boundary-value problem $(\ref{probl})$
 in the Sobolev space $H^1(\Omega_\varepsilon).$
Moreover, the following asymptotic estimates
\begin{equation}\label{t0}
    \forall \, m \in\Bbb{N} \ \ \exists \, {C}_m >0 \ \ \exists \, \varepsilon_0>0 \ \ \forall\, \varepsilon\in(0, \varepsilon_0) : \qquad \|u_\varepsilon-U_\varepsilon^{(m)}\|_{H^1(\Omega_\varepsilon)} \leq {C}_m \ \varepsilon^{2\, m+\frac{1}{2}}
\end{equation}
hold, where
$$
U^{(m)}_{\varepsilon}(x,y) = \omega_2(x) + \sum\limits_{k=1}^{m} \varepsilon^{2k-1}\left( \omega_{2k+1}(x) + \chi^0(x){N}_{2k-1}\left(\frac{x}{\varepsilon},\frac{y}{\varepsilon}\right) \right) +
\sum\limits_{k=1}^{m} \varepsilon^{2k}\Bigg( u_{2k}\left(x,\frac{y}{\varepsilon}\right) + \omega_{2k+2}(x)
$$
\begin{equation}\label{aaN}
+ \chi^0(x){N}_{2k}\left(\frac{x}{\varepsilon},\frac{y}{\varepsilon}\right) + \chi^-(x)\Pi^{(1)}_{2k}\Big(\frac{1+x}{\varepsilon},\frac{y}{\varepsilon}\Big) + \chi^+(x)\Pi^{(2)}_{2k}\Big(\frac{1-x}{\varepsilon},\frac{y}{\varepsilon}\Big) \Bigg), \quad (x,y)\in\Omega_\varepsilon,
\end{equation}
is the partial sum of $(\ref{asymp_expansion}).$
\end{theorem}

\begin{remark}
Hereinafter, all constants in inequalities are independent of  the parameter~$\varepsilon.$
\end{remark}

\begin{proof} Consider an arbitrary $m\in\Bbb{N}$. Substituting the partial sum  $ U^{(m)}_{\varepsilon}$
into equations and boundary conditions of the problem (\ref{probl}) and considering  relations
(\ref{main})--(\ref{new_prim-probl}) that satisfied by the coefficients of the series (\ref{asymp_expansion}), we find
$$
\Delta U^{(m)}_{\varepsilon}(x,y) + f\left(x,\frac{y}{\varepsilon}\right) =  \varepsilon^{2m}\partial_{xx}u_{2m}\left(x,\frac{y}{\varepsilon}\right) +
$$
$$
+ \sum\limits_{k=1}^{m} \varepsilon^{2k-1} \Big(
2 \varepsilon^{-1}\frac{d\chi^0}{dx}(x)\partial_{\xi} {N}_{2k-1}(\xi,\eta) + \frac{d^2\chi^0}{dx^2}(x) {N}_{2k-1}(\xi,\eta)
\Big)|_{\xi=\frac{x}{\varepsilon},\, \eta=\frac{y}{\varepsilon}} +
$$
$$
+
\sum\limits_{k=1}^{m} \varepsilon^{2k}\Big(2 \varepsilon^{-1} \frac{d\chi^0}{dx}(x) \partial_{\xi}{N}_{2k}(\xi,\eta) + \frac{d^2\chi^0}{dx^2}(x){N}_{2k}(\xi,\eta) +
$$
$$
+ 2 \varepsilon^{-1}\frac{d\chi^-}{dx}(x) \partial_{\xi}\Pi^{(1)}_{2k}(\xi,\eta) + \frac{d^2\chi^-}{dx^2}(x)\Pi^{(1)}_{2k}(\xi,\eta) +
$$
\begin{equation}\label{t1}
 2 \varepsilon^{-1}\frac{d\chi^+}{dx}(x)\partial_{\xi}\Pi^{(2)}_{2k}(\xi,\eta) + \frac{d^2\chi^0}{dx^2}(x)\Pi^{(2)}_{2k}(\xi,\eta) \Big)|_{\xi=\frac{1-x}{\varepsilon},\, \eta=\frac{y}{\varepsilon}} =: R^{(m)}_\varepsilon(x,y).
\end{equation}

Taking into account the exponential decreasing of  functions $\{{N}_{k},  \Pi^{(1)}_{k}, \Pi^{(2)}_{k}\}$ (see (\ref{inner_asympt}),
(\ref{as_estimates})), we conclude that
\begin{equation}\label{t2}
\exists\, \check{C}_m  \ \ \exists \, \varepsilon_0 \ \ \forall\, \varepsilon\in(0, \varepsilon_0) : \quad
\limsup\limits_{(x,y)\in\Omega^{(i)}_{\varepsilon}}\left|R_{\varepsilon}^{(m)}(x,y)\right| \leq \check{C}_m\varepsilon^{2m} \quad
(i=1, 2).
\end{equation}

It is easy  to check that the partial sum leaves the following residuals
$$
\begin{array}{rcll}
-\partial_y{U_\varepsilon^{(m)}}(x,y)|_{y=\pm{\varepsilon\frac{h_i}{2}}}        & = & {\varepsilon\varphi_\pm^{(i)}}(x), & x\in{I_i},
  \\[2mm]
U_\varepsilon^{(m)}(\pm1,y)  & = & 0, & y\in\Upsilon_\varepsilon^{(i)},
\\[2mm]
\partial_x{U_\varepsilon^{(m)}}(x,y)|_{x=0}     & = &
\varepsilon^{2m}\left(\partial_x u_{2m}^{(1)}\left(0,\frac{y}{\varepsilon}\right) -
\dfrac{h_2}{h_1+ h_2}\, d^+_{2m}\right)=:\overline{R}_\varepsilon^{(m)}(y), & y\in\Upsilon_\varepsilon^{(1)}\backslash\Upsilon_\varepsilon^{(2)},
\end{array}
$$
in the boundary conditions $(i=1,2)$ and the following ones
$$
\begin{array}{rcll}
[U_\varepsilon^{(m)}]|_{x=0}  & = & 0,   & y\in\Upsilon_\varepsilon^{(2)},
\\[2mm]
[\partial_x U_\varepsilon^{(m)}]|_{x=0}     & = &
\varepsilon^{2m}\left([\partial_x u_{2m}]|_{x=0} + \dfrac{h_2-h_1}{h_1+ h_2}\, d^+_{2m}\right)=:\widehat{R}_\varepsilon^{(m)}(y), & y\in\Upsilon_\varepsilon^{(2)}
\end{array}
$$
in transmission conditions.
Obviously that there exist positive constants $\overline{C}_m$ and  $\overline{\varepsilon}_0$ such that
\begin{equation}\label{t3}
     \forall\, \varepsilon\in(0, \overline{\varepsilon}_0):\quad \limsup\limits_{y\in\Upsilon_\varepsilon^{(1)}\backslash\Upsilon_\varepsilon^{(2)}}\left|\overline{R}_{\varepsilon}^{(m)}(y)\right| \leq \overline{C}_m\varepsilon^{2m}, \quad
\limsup\limits_{y\in\Upsilon_\varepsilon^{(2)}}\left|\widehat{R}_{\varepsilon}^{(m)}(y)\right| \leq \overline{C}_m \varepsilon^{2m}.
\end{equation}

Thus, the difference $W_\varepsilon := u_\varepsilon - U_\varepsilon^{(m)}$ satisfies the following system:
\begin{equation}\label{nevyazka}
\left\{\begin{array}{rcll}
-\Delta W_\varepsilon                                    & = & R_\varepsilon^{(m)}             & \mbox{in} \ \Omega_\varepsilon,                                                             \\[2mm]
\partial_y W_\varepsilon(x,\pm\varepsilon\frac{h_i}{2}) & = & 0,                              & x\in I_i, \, i=1,2,                                                                        \\[2mm]
W_\varepsilon(\pm 1,y)                                   & = & 0,                              & y\in\Upsilon_\varepsilon^{(i)}, \ i=1,2,                                                   \\[2mm]
- \partial_x W_\varepsilon (0,y)                           & = &  \overline{R}_\varepsilon^{(m)}, & y\in\Upsilon_\varepsilon^{(1)}\backslash\Upsilon_\varepsilon^{(2)},                        \\[2mm]
[W_\varepsilon]|_{x=0}                                   & = & 0,                              & y\in\Upsilon_\varepsilon^{(2)},                                                            \\[2mm]
[\partial_x W_\varepsilon]|_{x=0}                        & = & - \widehat{R}_\varepsilon^{(m)},  & y\in\Upsilon_\varepsilon^{(2)}.                                                            \\[2mm]
\end{array}\right.
\end{equation}
This means that the series (\ref{asymp_expansion}) is a formal asymptotic solution of problem (\ref{probl}).

From (\ref{nevyazka}) we derive the following integral relation:
$$
 \int \limits_{\Omega_\varepsilon} {|\nabla W_\varepsilon |}^2 dxdy =
 \int \limits_{\Omega_\varepsilon} R_\varepsilon^{(m)} \, W_\varepsilon \,dx dy -
 \int \limits_{\Upsilon_\varepsilon^{(1)}\backslash\Upsilon_\varepsilon^{(2)}}  \overline{R}_\varepsilon^{(m)}\, (W_\varepsilon |_{x=0}) \,dy + \int \limits_{\Upsilon_\varepsilon^{(2)}}  \widehat{R}_\varepsilon^{(m)}\, (W_\varepsilon |_{x=0}) \,dy.
$$
Now, using the Friedrichs inequality  and estimates  (\ref{t2}) and (\ref{t3}), we deduce from the previous integral equality the
following inequality
$$
\int \limits_{\Omega_\varepsilon} {|\nabla W_\varepsilon |}^2 dxdy
\leq \check{C}_m \sqrt{h_1+h_2} \ \varepsilon^{2m+\frac{1}{2}} \|W_\varepsilon\|_{L^2(\Omega_\varepsilon)} + \overline{C}_m \sqrt{h_1} \ \varepsilon^{2m+\frac{1}{2}} \|W_\varepsilon (0, \cdot)\|_{L^2(\Upsilon_\varepsilon^{(1)})} \leq
$$
$$
\leq {C}_m \, \varepsilon^{2m+\frac{1}{2}} \|\nabla W_\varepsilon\|_{L^2(\Omega_\varepsilon)},
$$
which means the asymptotic estimates (\ref{t0}) are satisfied. Those asymptotic estimates justify  the constructed asymptotics
and imply  that series $(\ref{asymp_expansion})$ is the asymptotic expansion for the solution of problem $(\ref{probl})$
 in $H^1(\Omega_\varepsilon).$
\end{proof}

\begin{corollary}\label{Corollary}
 For difference between solution $u_\varepsilon$ of problem (\ref{probl}) and solution $\omega_2$ of the homogenized problem (\ref{main})
the following asymptotic estimates
\begin{equation}\label{t4}
 \|u_\varepsilon - \omega_2\|_{L^2(\Omega_\varepsilon)} \leq {C}_0 \, \varepsilon^\frac{3}{2},\qquad
 \|u_\varepsilon - \omega_2\|_{H^1(\Omega_\varepsilon)} \leq {C}_0 \, \varepsilon,
\end{equation}
\begin{equation}\label{t5}
 \left\|u_\varepsilon-\omega_2-\varepsilon(\omega_{3} + \chi^0 {N}_{1})\right\|_{H^1(\Omega_\varepsilon)} \leq
 C_1 \, \varepsilon^\frac{3}{2},
\end{equation}
\begin{equation}\label{t6}
 \|E^{(i)}_\varepsilon(u_\varepsilon) - \omega_2\|_{L^2(I_i)} \leq {C}_2 \, \varepsilon,\qquad
 \|E^{(i)}_\varepsilon(u_\varepsilon) - \omega_2\|_{H^1(I_i)} \leq {C}_2 \, \varepsilon^\frac{1}{2}, \ \ i=1, 2,
\end{equation}
\begin{equation}\label{t7}
\max_{x\in \overline{I}_i} \left| E^{(i)}_\varepsilon(u_\varepsilon)(x) - \omega_2(x)\right|  \leq {C}_3 \, \varepsilon^\frac{1}{2}, \ \ i=1, 2,
\end{equation}
hold,  where
$$
E^{(i)}_\varepsilon(u_\varepsilon)(x) = \frac{1}{\varepsilon\, h_i}\int_{\Upsilon^{(i)}_\varepsilon} u_\varepsilon(x,y)\, dy, \quad i=1, 2.
$$
\end{corollary}

\begin{proof}
Since functions ${N}_{2},$ $\Pi^{(1)}_{2}$ and $\Pi^{(2)}_{2}$ exponentially decrease at infinity, from estimate (\ref{t0}) at $m=1$
we deduce inequality~(\ref{t5}):
$$
 \left\|u_\varepsilon-\omega_2-\varepsilon(\omega_{3} + \chi^0 {N}_{1})\right\|_{H^1(\Omega_\varepsilon)} \leq
$$
$$
\leq \left\|u_\varepsilon-U^{(1)}_{\varepsilon}\right\|_{H^1(\Omega_\varepsilon)} + \left\|U^{(1)}_{\varepsilon}-\omega_2-\varepsilon(\omega_{3} + \chi^0 {N}_{1})\right\|_{H^1(\Omega_\varepsilon)} \leq
$$
$$
\leq \widetilde{C}_1 \, \varepsilon^{\frac{5}{2}} +
 \varepsilon^{2} \left\|  u_{2} + \omega_{4} +
 \chi^0 {N}_{2} + \chi^- \Pi^{(1)}_{2} + \chi^+ \Pi^{(2)}_{2}  \right\|_{H^1(\Omega_\varepsilon)} \leq
$$
$$
\leq
\varepsilon^{2} \left\| \nabla_{x}\big( u_{2} + \omega_{4} +
 \chi^0 {N}_{2} + \chi^- \Pi^{(1)}_{2} + \chi^+ \Pi^{(2)}_{2} \big) \right\|_{L^2(\Omega_\varepsilon)}
+
$$
$$
+ \varepsilon^{2} \left\|  u_{2} + \omega_{4} +
 \chi^0 {N}_{2} + \chi^- \Pi^{(1)}_{2} + \chi^+ \Pi^{(2)}_{2}  \right\|_{L^2(\Omega_\varepsilon)}
+ \widetilde{C}_1 \, \varepsilon^{\frac{5}{2}}
 \leq
$$
$$
\leq
 \varepsilon^{\frac32} \sum_{i=1}^2 \Big(\int\limits_{I_i\times \Upsilon_i} |\partial_\eta u_2(x,\eta)|^2 dxd\eta\Big)^{\frac12}
 +
 $$
 $$
 + \varepsilon^{\frac{5}{2}} \sum_{i=1}^2 \Big( c_2 \|  \omega_{4}\|_{H^1(I_i)}
 +  \|  u_{2} \|_{L^2(I_i\times \Upsilon_i)} + \| \partial_x u_{2} \|_{L^2(I_i\times \Upsilon_i)}\Big) +
$$
$$
+\varepsilon^{2} \left(\| \nabla_{\xi \eta} {N}_{2}\|_{L^2(\Xi)}
+ \| \nabla_{\xi \eta} \Pi^{(1)}_{2}\|_{L^2((0,+\infty)\times \Upsilon_1)}
+\| \nabla_{\xi^* \eta} \Pi^{(2)}_{2} \|_{L^2((0,+\infty)\times \Upsilon_2 )}\right)
+
$$
\begin{equation}\label{t8}
+ \varepsilon^{3} \left(\|  {N}_{2}\|_{L^2(\Xi)}
+ \|  \Pi^{(1)}_{2}\|_{L^2((0,+\infty)\times \Upsilon_1)}
+\|  \Pi^{(2)}_{2} \|_{L^2((0,+\infty)\times \Upsilon_2 )}\right)
+ \widetilde{C}_1 \, \varepsilon^{\frac{5}{2}} \leq
 {C}_1 \, \varepsilon^\frac{3}{2}.
\end{equation}

Taking into account that
$\| \chi^0 {N}_{1}\|_{L^2(\Omega_\varepsilon)} \leq c_2 \, \varepsilon$
and $\| \chi^0 {N}_{1}\|_{H^1(\Omega^{(i)}_\varepsilon)} \leq c_3,\ i=1, 2,$
from (\ref{t5}) we get (\ref{t4}).

Using the Cauchy-Buniakovsky inequality, from (\ref{t4}) we derive inequalities (\ref{t6}).
Since the space $H^1(I_i)$ continuously embedded in $C(\overline{I}_i),$
 from the second inequality in (\ref{t6}) it follows inequality (\ref{t7}).
\end{proof}

\begin{remark}\label{zauv}
Summands of order $\varepsilon^\frac{3}{2}$
in the sixth line of $(\ref{t8})$ bring the main contribution to the determination of the constant $C_1$ in inequality $(\ref{t5}).$
Taking  the explicit form of the coefficients $u_2^{(i)}$ $($see $(\ref{solution_t}))$ into account, we can specify the dependence of this constant on the right-hand sides of problem $(\ref{probl}):$
\begin{multline}\label{t9}
 \sum_{i=1}^2 \Big(\int_{I_i\times \Upsilon_i} |\partial_\eta u_2(x,\eta)|^2 dxd\eta\Big)^{\frac12} \leq
 \\
 \leq \sum_{i=1}^2 \sqrt{h_i} \left( \sqrt{5 h_i} \|f\|_{L^2(I_i\times \Upsilon_i)} + 2\sqrt{2} \|\varphi_-^{(i)}\|_{L^2(I_i)} + \sqrt{6} \|\varphi_+^{(i)}\|_{L^2(I_i)} \right).
\end{multline}

Terms of order $\varepsilon^2$  in the eighth line of  $(\ref{t8})$ import next contribution to the constant $C_1.$
Let us estimate those terms. From the corresponding integral identity for the solution $N_2$ $($see $(\ref{integ_even})),$ we derive
$$
   \| \nabla_{\xi \eta} {N}_{2}\|_{L^2(\Xi)} \leq c(\delta)\|\Psi_2\|_{L^2(\Upsilon_2)} + \|\Psi^\prime_2\|_{L^2(\Upsilon_2)} +
   2 d^+_1 \frac{h_2 \sqrt{h_1} + h_1 \sqrt{h_2}}{h_1 + h_2} \leq
$$
$$
\leq 2 \sum_{i=1}^2 \Bigg( h_i \Big( c(\delta) \sqrt{2} h_i\sqrt{1+h_i^2} + \sqrt{5} \Big) \left( \|f\|_{L^2(I_i\times \Upsilon_i)} + \|\partial_x f\|_{L^2(I_i\times \Upsilon_i)} \right) +
$$
$$
   + \sqrt{h_i}\Big( c(\delta) h_i \sqrt{3+2h_i^2} + 2\sqrt{2} \Big)
   \left( \|\varphi_-^{(i)}\|_{L^2(I_i)} + \|\partial_x \varphi_-^{(i)}\|_{L^2(I_i)} \right) +
$$
$$
   + \sqrt{h_i}\Big( c(\delta) h_i \sqrt{3+2h_i^2} + \sqrt{6} \Big)
   \left( \|\varphi_+^{(i)}\|_{L^2(I_i)} + \|\partial_x \varphi_+^{(i)}\|_{L^2(I_i)} \right)
     \Bigg) +
$$
\begin{equation} \label{t10}
   + 2 d^+_1 \frac{h_2 \sqrt{h_1} + h_1 \sqrt{h_2}}{h_1 + h_2},
\end{equation}
where $\Psi_2(\eta)= u^{(1)}_2(0,\eta) - u^{(2)}_2(0,\eta).$
Similarly we can  estimate values $\| \nabla_{\xi \eta} \Pi^{(1)}_{2}\|_{L^2((0,+\infty)\times \Upsilon_1)}$ and
$\| \nabla_{\xi^* \eta} \Pi^{(2)}_{2} \|_{L^2((0,+\infty)\times \Upsilon_2 )}.$
\end{remark}
\begin{remark}
If $\varphi^{(i)}_\pm \equiv 0$ and function $f$ in the right-hand side of the differential equation of problem
$(\ref{probl})$ depends only on the variable $x,$ then all coefficients $\{u_{2k}\},$ $\{\Pi^{(1)}_{2k}\}$ and $\{\Pi^{(2)}_{2k}\}$
are equal to 0. In this case the asymptotic series $(\ref{asymp_expansion})$ has the following form
\begin{equation}\label{asymp_expansion1}
\omega_2(x) + \sum\limits_{k=1}^{+\infty} \varepsilon^{k}\Bigg( \omega_{k+2}(x) + \chi^0(x) {N}_{k}\left(\frac{x}{\varepsilon},\frac{y}{\varepsilon}\right) \Bigg), \quad (x,y)\in\Omega_\varepsilon,
\end{equation}
and asymptotic estimate $(\ref{t5})$ is of order $\varepsilon^2.$
Moreover, the value $\| \nabla_{\xi \eta} {N}_{2}\|_{L^2(\Xi)},$ which is now bounded by the value $2 d^+_1 \frac{h_2 \sqrt{h_1} + h_1 \sqrt{h_2}}{h_1 + h_2}$   $($see $(\ref{t10})),$ bring the main contribution to the constant $C_1$ in $(\ref{t5}).$
\end{remark}

\section{Discussion of the results}

{\bf 1.} Estimate (\ref{t5}) shows the structure of the corrector in the asymptotic approximation for the solution $u_\varepsilon$  of the problem (\ref{probl}). It (corrector) has the form
$$
\varepsilon \left(\omega_{3}(x)  + \chi^0(x) {N}_{1}\big(\tfrac{x}{\varepsilon}, \tfrac{y}{\varepsilon} \big) \right),
$$
and its gradient is equal
$$
\chi^0(x) \nabla_{\xi\eta}{N}_{1}(\xi, \eta)|_{\xi=\frac{x}{\varepsilon}, \eta= \frac{y}{\varepsilon}} +
\varepsilon \left(\omega^\prime_{3}(x)  + \big(\chi^0(x)\big)^\prime {N}_{1}\big(\tfrac{x}{\varepsilon}, \tfrac{y}{\varepsilon} \big) \right).
$$
Since function ${N}_{1}$  exponentially decreases  at infinity (see  (\ref{inner_asympt})), then
$$
\varepsilon \left\| \omega_{3}  + \chi^0\, {N}_{1} \right\|_{H^1(\Omega_\varepsilon)} \le C_0 \varepsilon.
$$
Thus, for this structure of thin cascade junctions  there are not any significant boundary effects
in a neighborhood the join zone  for  the solution $u_\varepsilon$ of problem (\ref{probl}). This also indicates second estimate in (\ref{t4}) and uniform pointwise estimate (\ref{t7}).

The results obtained give the right, in terms of practical application, to replace the complex boundary-value problem (\ref{probl}) with
the corresponding simple boundary-value problem (\ref{main}) with sufficient accuracy that measured by the parameter $\varepsilon$
characterizing the thickness.

Moreover, in this paper we  show how the constant $C_1$ in the main asymptotic estimate (\ref{t5}) depends on the right-hand sides of  problem (\ref{probl}) and on the geometrical parameters $h_1,$  $h_2$ and $d^+_1$ (see remark~\ref{zauv}).
Also it is possible to indicate how other constants  in the asymptotic estimates from Corollary~\ref{Corollary}
depend on these values. This fact makes it possible to directly use the asymptotic estimates for the approximation of solutions of boundary-value problems in thin cascade domains instead of numerical calculations.

\medskip

{\bf 2.} The method proposed in this paper  for the construction of  asymptotic expansions can be applied
 without substantial changes to the asymptotic study of boundary value-problems in thin cascade
  domains with more complex structures, namely, either
\begin{figure}[htbp]
\begin{center}
\includegraphics[width=13cm]{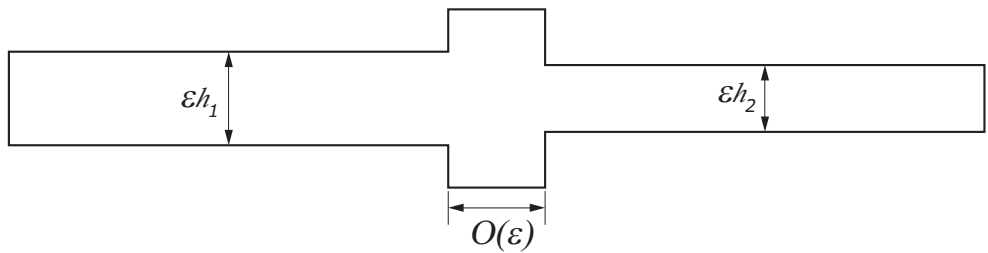}
\caption{Thin cascade domain with local widening}
\end{center}
\end{figure}
thin cascade domains with local widening (narrowing)  in a neighborhood of the join zone  (see Fig.~3),
\begin{figure}[htbp]
\begin{center}
\includegraphics[width=13cm]{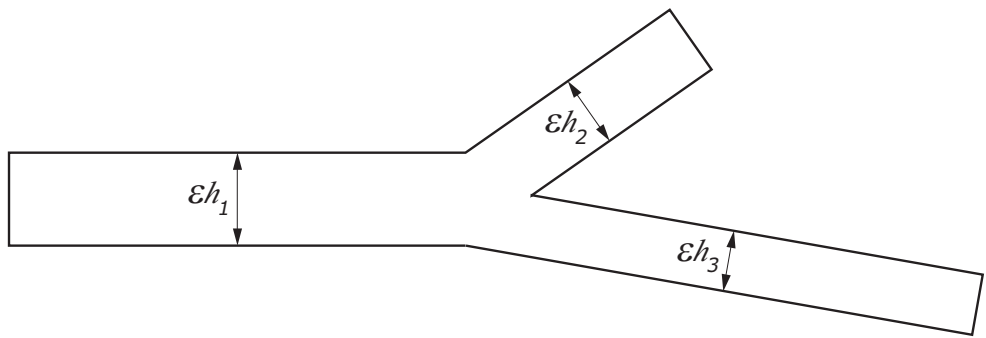}
\caption{Thin cascade domain of graph type}
\end{center}
\end{figure}
or thin cascade domains of graph type  (see Fig.~4),
\begin{figure}[htbp]
\begin{center}
\includegraphics[width=13cm]{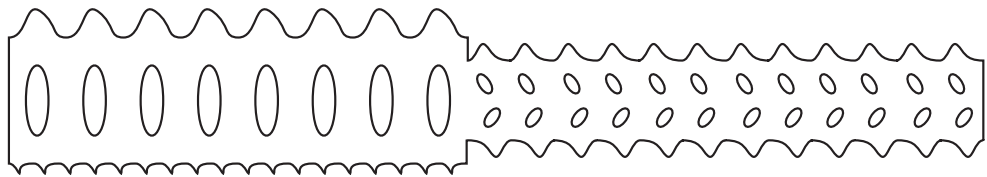}
\caption{Thin cascade perforated domain with rapidly varying thickness}
\end{center}
\end{figure}
or thin cascade perforated domains with rapidly varying thickness (Fig.~5).
We need to add series with rapidly oscillating coefficients (see \cite{Mel_Pop_MatSb,Mel_Pop_book})
to the regular part of the asymptotic expansion for solutions of boundary-value problems in
thin cascade perforated domains with rapidly varying thickness.


\end{document}